[1994/12/01]
\documentclass{ijmart}
\chardef\bslash=`\\ 

\usepackage{amssymb}
\usepackage{graphicx}  		
\usepackage{xypic}
\usepackage{bbm}
\usepackage{hyperref}
\usepackage{eucal}
\usepackage{mathrsfs}

\newtheorem{thm}{Theorem}[section]
\newtheorem{cor}[thm]{Corollary}
\newtheorem{lem}[thm]{Lemma}
\newtheorem{prop}[thm]{Proposition}

\numberwithin{equation}{section}

\theoremstyle{definition}

\newtheorem{defn}[thm]{Definition}
\newtheorem{remark}[thm]{Remark}
\newtheorem{example}[thm]{Example}

\theoremstyle{remark}

\newcommand{\CC}{\mathbbm{C}}
\newcommand{\RR}{\mathbbm{R}}

\newcommand{\ZZ}{\mathbbm{Z}}

\newcommand{\id}{\mathrm{id}}
\newcommand{\ad}{\mathrm{ad}}
\newcommand{\Ad}{\mathrm{Ad}}
\newcommand{\mat}{\mathrm{Mat}}

\newcommand{\Iso}{\mathrm{Iso}}

\newcommand{\GL}{\mathrm{GL}}

\newcommand{\hol}{\mathrm{Hol}}

\DeclareMathOperator{\im}{\mathrm{im}}
\DeclareMathOperator{\rk}{\mathrm{rank}}

\newcommand{\frg}{\mathfrak{g}}
\newcommand{\frh}{\mathfrak{h}}
\newcommand{\fra}{\mathfrak{a}}
\newcommand{\frn}{\mathfrak{n}}
\newcommand{\frb}{\mathfrak{b}}
\newcommand{\frz}{\mathfrak{z}}
\newcommand{\frw}{\mathfrak{w}}

\newcommand{\frr}{\mathfrak{r}}
\newcommand{\frv}{\mathfrak{v}}
\newcommand{\butt}{\mathfrak{b}_6}

\renewcommand{\d}{\mathrm{d}}

\newcommand{\eval}[2][\right]{\relax
  \ifx#1\right\relax \left.\fi#2#1\rvert}

\begin{document}
\title[Flat Homogeneous Spaces]{On the Geometry of Flat Pseudo-Riemannian Homogeneous Spaces}
\author{Wolfgang Globke}
\address{School of Mathematical Sciences\\
The University of Adelaide\\
SA 5005, Australia}
\email{\sf wolfgang.globke@adelaide.edu.au}


\begin{abstract}
Let $M=\RR^n_s/\Gamma$ be complete flat pseudo-Riemannian homogeneous mani\-fold,
$\Gamma\subset\Iso(\RR^n_s)$ its fundamental group and
$G$ the Zariski closure of $\Gamma$ in $\Iso(\RR^n_s)$.
We show that the $G$-orbits in $\RR^n_s$
are affine subspaces and affinely diffeomorphic to $G$ endowed
with the (0)-connection.
If the restriction of the pseudo-scalar product on $\RR^n_s$ to
the $G$-orbits is non-degenerate, then
$M$ has abelian linear holonomy.
If additionally $G$ is not abelian, then $G$
contains a certain subgroup of dimension $6$. In particular,
for non-abelian $G$,
orbits with non-degenerate metric can appear only if $\dim G\geq 6$.
Moreover, we show that $\RR^n_s$ is a trivial algebraic principal
bundle $G\to M\to\RR^{n-k}$.
As a consquence, $M$ is a trivial smooth bundle $G/\Gamma\to M\to\RR^{n-k}$ with compact fiber $G/\Gamma$.
\end{abstract}

\maketitle
\tableofcontents

\section{Introduction}

A geodesically complete flat pseudo-Riemannian homogeneous manifold $M$ 
with signature $(n-s,s)$ is of the
form $M=\RR^n_s/\Gamma$, where $\RR^n_s$ denotes $\RR^n$ endowed
with a non-degenerate symmetric bilinear form of signature
$(n-s,s)$, and
the fundamental group $\Gamma$ is a 2-step nilpotent unipotent
group of isometries.

These spaces were first studied by Joseph Wolf in
\cite{Wolf_1}, and
we review some of his results in Section \ref{sec_prelim}.
More recent studies by Oliver Baues and the author
\cite{baues,BG,globke,globke2} investigated among other things the
holonomy groups of these spaces.
For some time only $M$ with abelian linear holonomy
group (given by the linear parts of $\Gamma$) were known, and
it was unknown whether $M$ with non-abelian linear holonomy
existed.
For example, it was shown in \cite{BG} that for compact $M$
the linear holonomy group is always abelian.
The first example of a non-compact $M$ with non-abelian
linear holonomy in dimension $n=14$
was given by the author in his thesis
\cite{globke}, see also \cite{globke2}.

In the present article we study the geometry of geodesically
complete flat
pseudo-Riemannian homogeneous spaces. 
This is somewhat facilitated by considering the action of the
real Zariski closure $G$ of $\Gamma$ in $\Iso(\RR^n_s)$
rather than $\Gamma$ itself. The group $G$ is a unipotent
algebraic group and shares many algebraic properties with
$\Gamma$, as outlined in Section \ref{sec_prelim}.
This is essentially due to the fact that $G$ and $\Gamma$ have
the same transitive centralizer.

The structure of the $G$-orbits is investigated in Section 
\ref{sec_affine}. These orbits are affine subspaces of $\RR^n$
and by Proposition \ref{prop_flat_affine} are affinely diffeomorphic to $G$ endowed with the (0)-connection given by
$\nabla_X Y=\frac{1}{2}[X,Y]$.

In Section \ref{sec_nondeg} we assume the restriction of the pseudo-scalar product to the $G$-orbits to be non-degenerate.
In this case, some rather strong restrictions
on $G$ are imposed, as the quotients under $\Gamma$ of the
$G$-orbits are themselves compact flat
pseudo-Riemannian homo\-geneous spaces
(Proposition \ref{prop_nondeg_abhol}).
As aforementioned, the linear holonomy group of
$M$ is then neccessarily abelian.
Moreover it is shown in Proposition \ref{prop_butterfly_sub}
that if the fundamental group $\Gamma$ (or equivalently $G$)
is not abelian and if the restriction of
the metric to the $G$-orbits is non-degenerate,
then $G$ contains $\mathrm{H}_3\ltimes_{\Ad^*}\frh_3^*$ as a
subgroup,
where $\mathrm{H}_3$ is the Heisenberg group and $\frh_3^*$ its
Lie  algebra dual.
In particular, for non-abelian $G$,
spaces with non-degenerate orbits
can appear only if $\dim G\geq 6$.
We conclude in Corollary \ref{cor_M6L}
that if $\RR^n_s$ has $G$-orbits with
non-degenerate metric, every flat pseudo-Riemannian homogeneous
space $M=\RR^n_s/\Gamma$ fibers into copies of $G/\Gamma$ which
contain a pseudo-Riemannian submanifold
$(\mathrm{H}_3\ltimes_{\Ad^*}\frh_3^*)/\Lambda$, where
$\Lambda$ is a lattice in $\mathrm{H}_3\ltimes_{\Ad^*}\frh_3^*$
contained as a subgroup in $\Gamma$.
These results represent a modest first step towards a classification of those
flat pseudo-Riemannian homogeneous spaces with non-degenerate
orbits.

The orbits are the fibers of a principal $G$-bundle over the quotient
$\RR^n/G$. As $G$ is a unipotent algebraic group acting on
the affine space $\RR^n$, it begs the question of whether
this bundle is in fact a bundle of affine algebraic varieties.

In Section \ref{sec_bundle} we answer this question in
Theorem \ref{thm_principal_bundle}, stating that $\RR^n$ is
a trivial algebraic principal $G$-bundle $G\to\RR^n\to\RR^{n-k}$,
where $k=\dim G$.
As a direct consequence, $M$ is a smooth trivial bundle $G/\Gamma\to M\to\RR^{n-k}$
with compact fiber $G/\Gamma$
(Theorem \ref{thm_pb2}).
Here, the triviality is an established fact about algebraic
actions of unipotent groups, outlined in Appendix
\ref{sec_algebraic}.
But it is not clear a priori that the quotient $\RR^n/G$
exists as an affine variety.
This can be seen in Proposition \ref{prop_UUp_quotient}
by identifying $\RR^n$ with the algebraic
homogeneous space $U/U_p$ and relating the $G$-action to the
action on $U/U_p$ of a certain algebraic subgroup $U'$ of $U$.
By a result of Rosenlicht \cite{rosenlicht4}, $U/U'$ is an affine variety isomorphic to an affine space $\RR^{n-k}$, and it is also
a quotient for the $G$-action on $\RR^n$. So $\RR^n/G$ indeed
exists as an affine variety.

\subsection*{Acknowledgement:}
I wish to thank Oliver Baues
for reading an earlier version of this article and
suggesting improvements to the exposition.
During the preparation of this article I was
supported by the Australian Research Council grant {\small DP120104582}.

\section{Prerequisites on Algebraic Groups and Unipotent Groups}\label{sec_alg_grp}

References for this section are Borel \cite[Chapters I and II]{borel}
and Raghunathan \cite[Preliminaries and Chapter II]{ragh}.

\subsection{Algebraic Groups}

A \emph{linear algebraic group} $G$ is a subgroup of
$\GL_n(\CC)$ defined as the solution set of a system of
polynomial equations in the matrix coefficients.
In other words, $G$ is a subgroup which is also an affine 
variety in $\GL_n(\CC)$.
We will omit the term ``linear'' in the following, as all
groups in question are matrix groups.

If the equations defining $G$ consist of polynomials over a
subfield $K$ of $\CC$, the $G$ is a
\emph{$K$-defined} algebraic group.
The \emph{K-points of $G$} are  $G(K)=G\cap\GL_n(K)$.
Here, we are interested in the case $K=\RR$ and refer to the
$\RR$-points of some algebraic group as a
\emph{real algebraic group}.

\begin{example}
If $b$ is a symmetric bilinear form on $\RR^n$ represented by a
matrix $B$,
then its linear isometry group $\mathrm{O}(b)$ is a real
algebraic group given by
$g^\top B g=B$ which componentwise consists of real
polynomial equations in the matrix coefficients of $g$.
Similarly, the affine isometry group
$\Iso(b)=\mathrm{O}(b)\ltimes\RR^n$ is a real algebraic group,
albeit as a linear subgroup of $\GL_{n+1}(\RR)$ to accommodate
the translations.
\end{example}

\begin{example}
If $b$ is a symmetric bilinear form and
$H$ is any subgroup of $\Iso(b)$, then its centralizer
$\mathrm{Z}_{\Iso(b)}(H)$ is a real algebraic group. Its elements
$g$ must satisfy the equations $g h=h g$ for all $h\in H$, and
these are componentwise real polynomial equations in the
matrix coefficients of $g$ (as a matrix in
$\GL_{n+1}(\RR)$).
Note that this does not require $H$ to be an algebraic group
itself.
\end{example}

For an arbitrary subgroup $G\subset\GL_n(K)$ we call the 
smallest algebraic group $\overline{G}$ containing $G$ the
\emph{Zariski closure of $G$}.
Then $G$ is dense in $\overline{G}$ in the Zariski topology.
This fact often allows to extend properties of $G$ to
$\overline{G}$ by continuity in the Zariski topology.
If $G\subset\GL_n(\RR)$, then its \emph{real Zariski closure}
$\overline{G}(\RR)$ is a real algebraic group.

In Section \ref{sec_bundle} we employ some properties of
algebraic group actions:
A free action of an algebraic group on an affine variety $V$ is
\emph{principal} if the map
$\theta:G\times V\to V\times V$, $(g,x)\mapsto(g.x,x)$
is an algebraic isomorphism. This amounts to saying that the map
$\beta:V\times V\to G$, $(y,x)\mapsto g_{yx}$ with $g_{yx}.x=y$ is 
a morphism (it is well-defined because the action is free).

Let $V$ be an affine $K$-variety and $G$ a $K$-defined algebraic
group acting by $K$-morphisms on $V$.
An affine variety $W$ is called a \emph{geometric quotient}
if there exists a \emph{quotient morphism} $\pi:V\to W$, that is,
$\pi$ is a surjective and open $K$-morphism,
its fibers are the orbits of the $G$-action
and the pullback by $\pi$ induces an algbraic isomorphism
$\overline{K}[\pi(U)]\cong\overline{K}[U]^G$ for each open subset
$U\subseteq V$, where $\overline{K}[U]^G$
is the ring of $G$-invariant regular functions on $U$ and
$\overline{K}$ is the algebraic closure of $K$
(these quotients are called ``geometric'' to distinguish them from
a different concept of quotients called ``algbraic'' which is of 
no interest here).

\subsection{Unipotent Groups}

In this article we are mostly concerned with \emph{uni\-potent}
groups. These are matrix groups $G$ in which
every element $g\in G$ is a unipotent matrix, that is,
$(g-I)^k=0$ holds for some $k>0$,
where $I$ denotes the identity matrix.
If $G$ is unipotent, its Zariski closure and real Zariski closure
(if applicable) are also unipotent by continuity.

\begin{example}
Let $G$ be a unipotent Lie group with Lie algebra $\frg$.
Then the exponential map $\exp:\frg\to G$ is a polynomial
diffeomorphism and thus
$G$ is an algebraic group. In particular, $G$ coincides with
its Zariski closure.
\end{example}

If $ G $ is a unipotent Lie group, then there exists a connected normal
Lie subgroup $ H $ of codimension $1$ and a subgroup
$ A \cong\mathrm{G}_{\mathrm{a}}$ (the additive group of the field of definition)
such that $ G $ is the semidirect product
\begin{equation}
 G  =  A  \cdot  H .
\label{eq_malcev_opg}
\end{equation}
See Onish\-chik and Vinberg \cite[Chapter 2,
Section 3.1]{OV} for details.

Applying the decomposition (\ref{eq_malcev_opg}) repeatedly and
exploiting the fact that $\exp:\frg\to G$ is a diffeomorphism,
we find the existence of a \emph{Malcev basis} $X_1,\ldots,X_k$
of $\frg$, which is a basis such that every $g\in G$ can be
written as
\begin{equation}
g=g(t_1,\ldots,t_k)=\exp(t_1 X_1+\ldots+t_k X_k)
\label{eq_exp_coord}
\end{equation}
for unique real parameters $t_1,\ldots,t_k$, the
\emph{exponential coordinates}.

If $\Gamma$ is a discrete subgroup of $G$, then $G$ is the Zariski
closure of $\Gamma$ if and only if $\Gamma$ is a lattice
(meaning $G/\Gamma$ is compact).
In this case, the dimension of $G$ equals the \emph{rank} (or
\emph{Hirsch length}) of $\Gamma$ (see Raghunathan \cite[Theorem 2.10]{ragh}).

\section{Prerequisites on Flat Homogeneous Spaces}\label{sec_prelim}

Let $\RR^n_s$ denote the space $\RR^n$ endowed with a non-degenerate symmetric
bilinear form of signature $(n-s,s)$ and $\Iso(\RR^n_s)$ its group of isometries. We will assume $n-s\geq s$ throughout.

Affine maps of $\RR^{n}$ are written 
as $\gamma= (I+A,v)$, where $I+A$ is the linear part ($I$ the
identity matrix), and $v$ the translation part.
The image space of $A$ is denoted by $\im A$.

Let $M$ denote a complete flat  pseudo-Riemannian 
homogeneous manifold.
Then $M$ is of the form $M=\RR^n_s/\Gamma$
with fundamental group $\Gamma\subset\Iso(\RR^n_s)$.
In particular, $\Gamma$ 
acts without fixed points on $\RR^n_s$.
Homogeneity is determined by the
action of the
centralizer $\mathrm{Z}_{\Iso(\RR^n_s)}(\Gamma)$ of $\Gamma$ in
$\Iso(\RR^n_s)$ (see \cite[Theorem 2.4.17]{Wolf_buch}):

\begin{thm}\label{thm_transitive}
Let $p:\tilde{M}\to M$ be the universal pseudo-Riemannian covering of $M$
and let $\Gamma$ be the group of deck transformations.
Then $M$ is homogeneous if and only if
$\mathrm{Z}_{\Iso(\tilde{M})}(\Gamma)$ acts transitively on $\tilde{M}$.
\end{thm}

Wolf studied subgroups $G\subset\Iso(\RR^n_s)$ such that $\mathrm{Z}_{\Iso(\RR^n_s)}(G)$
acts transitively on $\RR^n_s$.
To avoid repetitions we shall call them
\emph{Wolf groups} throughout this article.

\begin{remark}
By continuity, the real Zariski closure of a Wolf
group $G$ is also a Wolf group.
Moreover, $G$ is abelian if and only if its Zariski closure
is.
\end{remark}

\begin{prop}
$G$ acts freely and properly on $\RR^n_s$.
\end{prop}
\begin{proof}
$G$ acts freely because its centralizer acts transitively.
Properness was proved in a previous article by Baues and Globke \cite[Proposition 7.2]{BG}.
\end{proof}

For the following properties see Wolf's book \cite[Chapter 3]{Wolf_buch}:

\begin{lem}\label{lem_wolf1}
$G$ consists of affine transformations $g=(I+A,v)$, where $A^2=0$,
$v\perp\im A$ and $\im A$ is totally isotropic.
\end{lem}

\begin{lem}\label{lem_wolf2}
For $g_i=(I+A_i,v_i)\in G$, $i=1,2,3$, we have
$A_1 A_2=-A_2 A_1$,
$A_1 v_2=-A_2 v_1$,
$A_1 A_2 A_3=0$
and
$[g_1,g_2] = (I+2A_1 A_2, 2A_1 v_2)$.
\end{lem}

\begin{lem}\label{lem_wolf3}
If $g=(I+A,v)\in G$, then $\langle Ax,y\rangle=-\langle x,Ay\rangle$,
$\im A=(\ker A)^\bot$, $\ker A=(\im A)^\bot$ and $Av=0$.
\end{lem}

\begin{remark}\label{rem_unipotent}
Let $g=(I+A,v)\in G$ and $X=(A,v)$.
It follows from $A^2=0$ and $Av=0$ that $X^2=(A,v)^2=0$, so
the elements of $G$ are unipotent.
Moreover,
\begin{equation}
\exp(X)=\exp(A,v)=(I+A,v)=g.
\label{eq_exp}
\end{equation}
\end{remark}

\begin{thm}\label{thm_wolf1}
$G$ is $2$-step nilpotent (meaning $[G,[G,G]]=\{\id\}$)
and uni\-potent.
\end{thm}

If $\Gamma$ is the fundamental group of $M$, it also determines the holonomy of $M$:
For $\gamma=(I+A,v)\in\Gamma$, set $\hol(\gamma)=I+A$ (the linear component of
$\gamma$). We write $A=\log(\hol(\gamma))$.
The \emph{linear holonomy group} of $\Gamma$ (or of $M$) is
\[
\hol(\Gamma)=\{\hol(\gamma) \mid \gamma\in \Gamma\}.
\]
This name is justified by the following observation:
Let $x\in M$ and $\gamma\in\pi_1(M,x)$ be a loop. Then $\hol(\gamma)$
corresponds to the parallel transport $\tau_x(\gamma):\mathrm{T}_xM\rightarrow\mathrm{T}_xM$ in a natural way, see \cite[Lemma 3.4.4]{Wolf_buch}.

\section{The Affine Structure on the Orbits}\label{sec_affine}

Let $G \subset\Iso(\RR^{n}_{s})$ be a Zariski-closed Wolf group of 
dimension $k$
and $\frg$ its Lie algebra.
We study the affine structure on the orbits
$F_p=G.p$ of the $G$-action on
$\RR^{n}$.
The \emph{orbit map} at $p$ is $\theta:G\to F_p$, $g\mapsto g.p$.


\begin{prop}\label{prop_orbit_affine}
Let $X_1,\ldots,X_k$ be a Malcev basis of $\frg$, with $X_i=(A_i,v_i)$.
For every $p\in\RR^n$, set $b_i(p)=A_i p+v_i$.
Then the orbit $F_p$ is the affine subspace
\begin{equation}
F_p = p + \mathrm{span}\{ b_1(p),\ldots,b_k(p) \}
\label{eq_orbit_affine}
\end{equation}
of dimension $\dim F_p=k$.
\end{prop}
\begin{proof}
By (\ref{eq_exp}),
$\exp(X)=\exp(A,v)=(I+A,v)$ for all $X=(A,v)\in\frg$.

Now let $t_1,\ldots,t_k$ be the exponential coordinates
(\ref{eq_exp_coord}) for $G$.
Then
\begin{align*}
g(t_1,\ldots,t_k).p
&=
p + t_1 (A_1 p + v_1) +\ldots+ t_k (A_k p + v_k) \\
&= p + t_1 b_1(p) +\ldots+ t_k b_k(p)
\end{align*}
parameterizes an affine subspace in $\RR^n$.
The assertion on the dimension is standard, taking into account that
$G $ acts freely.
\end{proof}

Since $G $ acts freely, the natural affine connection $\nabla$ on 
the affine space $F_p$ pulls back to a flat affine connection $\nabla$
on $G $ through the orbit map.

\begin{remark}\label{rem_affine_map}
Because $X^2=0$ for all $X\in\frg\subset\mat_{n+1}(\RR)$, $\exp(X)=I+X$.
So $G =I+\frg$ is an affine subspace of $\mat_{n+1}(\RR)$
which therefore has a natural affine connection $\nabla^{G }$.
This connection is left-invariant because left-multiplication is linear on
$\mat_{n+1}(\RR)$.
The orbit map $\theta:G \to F_p$, $I+X\mapsto(I+X).p$ is an affine map
(if one chooses $I\in G $ and $p\in F_p$ as origins, the linear part of $\theta$
is $X\mapsto X\cdot p$ and the translation part is $+p$).
It is also a diffeomorphism onto $F_p$ because the action is free and $\exp$
is a diffeomorphism.
\end{remark}

From the above we immediately obtain:

\begin{cor}
$(G ,\nabla^{G })$ is affinely diffeomorphic to $(G ,\nabla)$.
\end{cor}

It follows from Lemma \ref{lem_wolf2} and (\ref{eq_exp})
that $XY=\frac{1}{2}[X,Y]$ for all $X,Y\in\frg$.
There exists a
bi-invariant flat affine connection $\tilde{\nabla}$ on $G $ given by
\begin{equation}
(\tilde{\nabla}_X Y)_g = \frac{1}{2}[X,Y]_g = g X_I Y_I,
\label{eq_canonical_conn}
\end{equation}
where $X,Y$ are left-invariant vector fields on $G $ and
$X_I, Y_I\in\frg$ their respective values at the identity $I$.
In fact, $\tilde{\nabla}$ is bi-invariant because $[X,Y]$ is
$\Ad(G)$-invariant, and it is flat because $\frg$ is 2-step nilpotent.
The connection $\tilde{\nabla}$ is called the \emph{(0)-connection} on $G$. 

\begin{prop}\label{prop_flat_affine}
The (0)-connection $\tilde{\nabla}$ on $G $ coincides with the
flat affine connection $\nabla$ on $G $.
\end{prop}
\begin{proof}
As both connections are left-invariant, is suffices to show that they coincide
on left-invariant vector fields.
Expressed in matrix terms, left-invariance for vector fields means
$X_g=g X_I$ for all
$X\in\frg$, $g\in G $. So for all $X,Y\in\frg$,
\begin{align*}
(\nabla_X Y)_g &= \lim_{t\to 0}\frac{Y_{g\exp(tX)}-Y_g}{t}
=\lim_{t\to 0}\frac{g(I+tX_I)Y_I-gY_I}{t} \\
&= \lim_{t\to 0} g X_I Y_I = g X_I Y_I \\
&= (\tilde{\nabla}_X Y)_g,
\end{align*}
where the first and last equality hold by definition.
\end{proof}

The metric $\langle\cdot,\cdot\rangle$ on the orbit $F_p$ pulls back
to a field $(\cdot,\cdot)$ of (possibly de\-generate) left-invariant symmetric
bilinear forms on $G $ which is parallel with respect to $\nabla$.
By abuse of language we call $(\cdot,\cdot)$ the \emph{orbit metric}
on $G $.
Since all orbits are isometric, the pair $(\nabla,(\cdot,\cdot))$ does not depend on
$p$ and is an invariant of $G \subset\Iso(\RR^{n}_{s})$.

\begin{prop}\label{prop_biinvariant_fibermetric}
The orbit metric $(\cdot,\cdot)$ is bi-invariant on $G $, that is
\[
([X,Y],Z) = - (Y,[X,Z])
\]
for all left-invariant vector fields $X,Y,Z$.
\end{prop}
If $(\cdot,\cdot)$
is non-degenerate this is well-known (O'Neill \cite[Proposition 11.9]{oneill}).
\begin{proof}
Let $X,Y,Z$ be a left-invariant vector fields on $G$.
Fix $g\in G$  and let
$X^+,Y^+,Z^+$ be the (right-invariant) Killing fields on $G $ such that
$X^+_g=X_g$, $Y^+_g=Y_g$, $Z^+_g=Z_g$.
The flow $\psi_t(h)$ of $X^+$ at $h\in G$ is given by
\[
\psi_t(h)=\exp(tgX_I g^{-1})h=L_{\exp(tgX_I g^{-1})}(h),
\]
in particular
$\psi_t(g)=g\exp(tX_I)$.
Because $\exp(- gX_Ig^{-1}) g\exp(X_I)
= g$ and $Y$ is left-invariant,
\begin{align*}
[X^+,Y]_g &= \frac{\d}{\d t}\Bigl|_{t=0}\d\psi_{-t} Y_{\psi_t(g)}
= \frac{\d}{\d t}\Bigl|_{t=0}\d L_{\exp(-tgX_Ig^{-1})} Y_{g\exp(tX_I)}
= \frac{\d}{\d t}\Bigl|_{t=0} Y_g
= 0.
\end{align*}
This implies
\[
(\nabla_{X^+} Y)_g = (\nabla_Y X^+)_g.
\]
$\nabla_X Y$ is tensorial in $X$, so
\[
(\nabla_X Y)_g = (\nabla_{X^+} Y)_g = (\nabla_Y X^+)_g
=(\nabla_{Y^+} X^+)_g.
\]
$X^+,Y^+$ are pullbacks of Killing fields on $\RR^n_s$
restricted to $F_p$,
so \cite[Proposition 3.10 (1)]{baues} gives $\nabla_{Y^+}X^+=-\nabla_{X^+}Y^+$.
Then
\[
(\nabla_X Y)_g = - (\nabla_{X^+} Y^+)_g.
\]
Now it follows from (\ref{eq_canonical_conn})
and the computations above that
\begin{align*}
( \frac{1}{2}[X,Y],Z )_g + ( Y,\frac{1}{2}[X,Z])_g
&= ( \nabla_X Y,Z^+ )_g + ( Y^+,\nabla_X Z)_g \\
&= ( -\nabla_Y X,Z^+ )_g + ( Y^+,-\nabla_Z X)_g \\
&= ( \nabla_{Y^+}X^+,Z^+ )_g + ( Y^+,\nabla_{Z^+}X^+)_g \\
&= 0.
\end{align*}
$(\cdot,\cdot)$ is a tensor, so we can replace $Z$ by $Z^+$
and $Y$ by $Y^+$ in the first line.
The last equality holds because $\nabla X^+$ is
skew-symmetric with respect to $(\cdot,\cdot)$ for a Killing field $X^+$ (see \cite[Subsection 3.4.1]{baues}).

As $g$ was arbitrary, $([X,Y],Z)=-(Y,[X,Z])$ holds everywhere.
So $(\cdot,\cdot)$ is bi-invariant.
\end{proof}

\begin{remark}\label{rem_orbit_metric}
The orbit metric $(\cdot,\cdot)$ on $G $ induces a symmetric 
bilinear  form $(\cdot,\cdot)$ on $\frg$ which
satisfies $([X,Y],Z)=-(Y,[X,Z])$ for all $X,Y,Z\in\frg$.
Such a form is called \emph{invariant}.
The \emph{radical} of $(\cdot,\cdot)$ in $\frg$ is the
subspace
\[
\frr = \{ X\in\frg \mid (X,\frg)=\{0\} \}.
\]
The radical $\frr$ is an ideal in $\frg$ due to the invariance of $(\cdot,\cdot)$.
\end{remark}

\begin{lem}\label{lem_commutator_isotropic}
The commutator subalgebra $[\frg,\frg]$ is a totally isotropic subspace
of $\frg$ with respect to $(\cdot,\cdot)$. The center $\mathfrak{z}(\frg)$ is
orthogonal to $[\frg,\frg]$.
\end{lem}
\begin{proof}
$\frg$ is 2-step nilpotent.
So
\[
([X_1,X_2],[X_3,X_4])=-(X_2,[X_1,[X_3,X_4]])=-(X_2,0)=0
\]
for all $X_i\in\frg$.
If $Z\in\mathfrak{z}(\frg)$, then $(Z,[X_1,X_2])=-([X_1,Z],X_2)=0$.
\end{proof}

\begin{cor}\label{cor_Z_Zstar}
Assume there exists $Z\in[\frg,\frg]$ and $Z^*\in\frg$ such that
$(Z,Z^*)\neq 0$. Then $Z^*\not\in\frz(\frg)$.
\end{cor}

\begin{lem}\label{lem_Z_perp_span}
If $Z=[X,Y]$, then $Z\perp\mathrm{span}\{X,Y,Z\}$.
\end{lem}
\begin{proof}
Use invariance and 2-step nilpotency.
\end{proof}

\begin{example}
Let $\Gamma\cong\mathrm{H}_3(\ZZ)$ with Zariski closure $G \cong\mathrm{H}_3$ (the discrete and real Heisenberg groups, respectively).
Assume $M=\RR^{n}_{s}/\Gamma$ is a flat pseudo-Riemannian homogeneous manifold.
The orbit metric induced on $G $ is degenerate, and $\mathfrak{z}(\frg)\subset\frr$:
Let $X,Y$ denote the Lie algebra generators of $\mathfrak{h}_3$ and $Z=[X,Y]$.
By Lemma \ref{lem_Z_perp_span}, $Z\in\frr$. So the non-degenerate case
is excluded.

The possible signatures are $(0,0,3)$, $(1,0,2)$, $(1,1,1)$ and $(2,0,1)$.
\end{example}

\begin{remark}
One easily realizes the signatures in the above example
by modi\-fying the metric in \cite[Example 6.4]{BG} accordingly.
\end{remark}

\section{Non-Degenerate Orbits}\label{sec_nondeg}

As before, let $M=\RR^n_s/\Gamma$ and $G$ the Zariski closure of $\Gamma$
with Lie algebra $\frg$.
If $\frg$ is not abelian and the orbit metric $(\cdot,\cdot)$
is non-degenerate then there are
some strong constraints on the structure of $\frg$.

\begin{prop}\label{prop_nondeg_abhol}
If the orbit metric on $G $ is non-degenerate,
then the quotient $F_p/\Gamma$ of each orbit is itself a compact
flat pseudo-Riemannian homogeneous space.
In particular, the linear holonomy group of $M$ is abelian.
\end{prop}
\begin{proof}
The orbits $F_p$ are affine subspaces of $\RR^n$ and isometric to $G$.
So $F_p/\Gamma=G /\Gamma$ is a compact
flat pseudo-Riemannian space.
Moreover, in Proposition \ref{prop_RFp} we will see that a group
centralizing $\Gamma$ acts transitively on $F_p$, so $F_p/\Gamma$
is homogeneous.
It was shown in \cite[Corollary 3.3]{BG} that the linear holonomy of $G$ is abelian.
\end{proof}

Additionally, $\frg$ must contain a subalgebra of a certain type.

\begin{defn}\label{def_butterfly}
A \emph{butterfly algebra}
$\butt$ is a 2-step nilpotent
Lie algebra of dimension $6$ endowed with an invariant
bilinear form $(\cdot,\cdot)$ such that there exists
$Z\in[\butt,\butt]$ with $Z\not\in\frr$.
\end{defn}

The naming in Definition \ref{def_butterfly} will become clear
after the proof of the following proposition:

\begin{prop}\label{prop_butterfly}
A butterfly algebra $\butt$ admits a vector space decompo\-sition
\[
\butt = \frv \oplus [\butt,\butt],
\]
where the subspaces $\frv$ and $[\butt,\butt]$ are totally isotropic and
dual to each other. In particular, $(\cdot,\cdot)$ is non-degenerate of
signature $(3,3)$.
\end{prop}
\begin{proof}
Let $X,Y\in\butt$ such that $Z=[X,Y]\neq 0$.
By Lemma \ref{lem_commutator_isotropic} $[\butt,\butt]$ is totally isotropic.
By assumption there exists $Z^*\in\butt\setminus[\butt,\butt]$ such that
\[
(Z,Z^*)=1.
\]
As a consequence of Lemma \ref{lem_Z_perp_span}, $X,Y,Z^*$ are linearly
independent, so they span a 3-dimensional subspace $\frv$.
Since $(\cdot,\cdot)$ is invariant,
\begin{align*}
1&=([X,Y],Z^*)=(Y,[Z^*,X]),\\
1&=([X,Y],Z^*)=(X,[Z^*,Y]).
\end{align*}
Set $X^*=[Z^*,Y]$ and $Y^*=[Z^*,X]$. Lemma \ref{lem_Z_perp_span}
further implies that $X^*,Y^*,Z$ are linearly independent, hence span a
3-dimensional subspace $\frw$ of $[\butt,\butt]$.

Since $\butt$ is 2-step nilpotent $[\butt,\butt]\subset\mathfrak{z}(\butt)$.
But $\frv\cap\mathfrak{z}(\butt)=\{0\}$, so it follows from dimension reasons
that $\frw=[\butt,\butt]=\mathfrak{z}(\butt)$. By construction also
\[
X\perp\mathrm{span}\{Y^*,Z\},\quad Y\perp\mathrm{span}\{X^*,Z\}.
\]
After a base change we may assume that $\frv$ is a dual space to $[\butt,\butt]$.
\end{proof}

The bases $\{X,Y,Z^*\}$ and $\{X^*,Y^*,Z\}$ from the proof above are dual bases
to each other. The following diagram describes the relations between these
bases, where solid lines from two elements indicate a commutator and dashed
 lines indicate duality between the corresponding elements:
\begin{center}
	\includegraphics{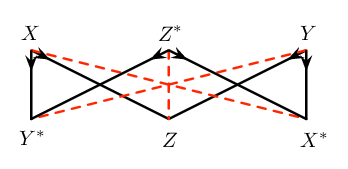}
\end{center}
This explains the name.
In particular, the following corollary justifies to speak of ``the''
butterfly algebra:\footnote{Some people misguidedly believe that the name
\emph{bat algebra} would be more apt.}

\begin{cor}
Any two butterfly algebras are isometric and isomorphic as Lie algebras.
\end{cor}
\begin{proof}
In every butterfly algebra one can find a basis as in the proof of
Proposition \ref{prop_butterfly}. Mapping such a basis of one butterfly algebra
to the corres\-ponding elements of another yields an isometric
Lie algebra isomorphism.
\end{proof}

From the above we conclude:

\begin{cor}
Let $\mathfrak{h}_3$ denote the Heisenberg algebra and $\ad^*$ its
coadjoint re\-presentation.
Up to isometric iso\-morphism, the butterfly algebra is
\begin{equation}
\butt = \mathfrak{h}_3\oplus_{\ad^*}\mathfrak{h}_3^*
\label{eq_b6}
\end{equation}
with $(\cdot,\cdot)$ given by
\begin{equation}
( X+X^*, Y+Y^* )
= X^*(Y)+Y^*(X)
\label{eq_b6ip}
\end{equation}
for $X,Y\in\mathfrak{h}_3$, $X^*,Y^*\in\mathfrak{h}_3$.
\end{cor}

So $\butt$ is the Lie algebra of the Lie group
$\mathrm{B}_6=\mathrm{H}_3\ltimes_{\Ad^*}\frh_3^*$,
where $\mathrm{H}_3$ denotes the Heisenberg group.
For a lattice $\Lambda$ in $\mathrm{B}_6$ denote the space
$\mathrm{B}_6/\Lambda$ by $M_6(\Lambda)$.
Baues \cite[Example 4.3]{baues} gave the spaces $M_6(\Lambda)$
as a class of compact homogeneous pseudo-Riemannian manifolds
of dimension $6$ with non-abelian fundamental group
and noted that these are the only possible examples.
This is the minimal dimension for such examples with non-abelian
fundamental group
according to \cite[Corollary 4.9]{baues}. We can rederive this
result as a consequence of the following:

\begin{prop}\label{prop_butterfly_sub}
If $G $ is not abelian and the orbit metric on $G $ is non-degenerate,
then $G$ contains a butterfly subgroup $B\cong\mathrm{B}_6$.
In particular, $\dim G \geq 6$.
Moreover, $B$ can be chosen such that it contains
a lattice $\Lambda=B\cap\Gamma$.
\end{prop}
\begin{proof}
Let $\frg$ denote the Lie algebra of $G$ and let
$\frg_{\Gamma}=\log(\Gamma)$ denote the discrete subset of $\frg$
which maps to $\Gamma$ under the exponential map.

There exists a Malcev basis of $\frg$ contained in
$\frg_{\Gamma}$ and in particular, as $\frg$ is not abelian,
we can find $X,Y,Z\in\frg_{\Gamma}$ such that $[X,Y]=Z\neq0$.
This is evident from the construction of a Malcev basis
strongly based on $\Gamma$ in Corwin and Greenleaf's proof of
\cite[Theorem 5.1.6]{corwin}.

Recall that $\frg$ is isometric to the $G$-orbits when
endowed with the metric $(\cdot,\cdot)$ induced by the orbit
metric (Remark \ref{rem_orbit_metric}).
Thus there exists $Z^*\in\frg_{\Gamma}$ such that
$(Z,Z^*)\neq0$,
otherwise the Malcev basis would span a space orthogonal 
to $Z$, which contradicts the non-degeneracy of $(\cdot,\cdot)$.
By Corollary \ref{cor_Z_Zstar}, $Z^*\not\in\frz(\frg)$.

We can then complete this to a basis $\{X,Y,Z,X^*,Y^*,Z^*\}$ of a
butterfly subalgebra as in the proof of Proposition
\ref{prop_butterfly}.
Here, $X^*,Y^*$ are elements of $\frg_{\Gamma}$ because they 
arise as commutators of elements of $\frg_{\Gamma}$ and
for 2-step nil\-potent Lie groups $\exp([X,Y])=[\exp(X),\exp(Y)]$ holds by the Baker-Campbell-Hausdorff formula,
so $\frg_{\Gamma}$ is closed under commutators.

Now define $\Lambda$ as the discrete subgroup generated by the
exponentials of the basis above.
By construction $\Lambda\subseteq\Gamma$
and the Zariski closure of $\Lambda$ is a butterfly subgroup
$B\subseteq G$.
\end{proof}

\begin{cor}\label{cor_M6L}
Under the assumptions of Proposition \ref{prop_butterfly_sub},
$M$ is the disjoint union of compact manifolds $G/\Gamma$,
each of which contains a submanifold isometric to $M_6(\Lambda)$
for a certain lattice $\Lambda$ in $\mathrm{B}_6$.
\end{cor}
\begin{proof}
$\RR^n$ is the union of disjoint $G$-orbits $F_p$, each
isometric to $G$.
Hence $M$ is the disjoint union of the $F_p/\Gamma$
which are isometric to $G/\Gamma$.
They are compact, as $G$ is the Zariski closure of $\Gamma$
(see Raghunathan \cite[Theorem 2.1]{ragh}).

If $\Lambda$ is the lattice in the butterfly subgroup $B$
from Proposition \ref{prop_butterfly_sub},
then $B$ is isometric to $\mathrm{B}_6$ and
\[
B/(\Gamma\cap B)=B/\Lambda=M_6(\Lambda).
\]
As $B$ is isometric to an affine subspace $B.p\subseteq F_p$,
it follows that $F_p/\Gamma$ contains a submanifold
(isometric to) $M_6(\Lambda)$.
\end{proof}

In Section \ref{sec_bundle} we  investigate the structure of $M$
as a fiber bundle with compact fiber $G/\Gamma$.

\begin{remark}
The concept of a butterfly algebra can be generalized to higher
dimensions by setting
\[
\frb_{\frn,\omega} = \frn\oplus_{\omega}\frn^*,
\]
where $\frn$ is a $k$-dimensional 2-step nilpotent Lie algebra
and $\omega$ a 2-cocycle for the coadjoint action
(\cite[Section 5.3]{baues}),
and defining a pseudo-scalar product as in (\ref{eq_b6ip}).
Let $\mathrm{B}_{\frn,\omega}$ denote a simply connected Lie group
with Lie algebra $\frb_{\frn,\omega}$.
It is known that every compact flat pseudo-Riemannian homogeneous
space $M$ of split signature $(k,k)$
can be realized as a quotient $\mathrm{B}_{\fra,\omega}/\Gamma$, 
where $\fra$ is an abelian Lie algebra
(see Baues and Globke \cite[Section 3]{BG}).
An interesting open question along these lines
is whether a compact flat pseudo-Riemannian homogeneous
space $M$ with split signature $(k,k)$
and non-abelian fundamental group $\Gamma$
can be realized as $\mathrm{B}_{\frn,\omega}/\Gamma$, where
$\frn$ is some non-abelian 2-step nilpotent Lie algebra.
\end{remark}

\section{A Trivial Bundle}\label{sec_bundle}

Let $G$ be an algebraic Wolf group and
let $ L $ denote its centralizer in $\Iso(\RR^{n}_{s})$.
In this section we will be mostly concerned with the properties
of $G$ as an algebraic group acting on the affine space $\RR^n$,
so we drop the index $s$ from $\RR^n_s$.

\begin{lem}
The $G$-action on $\RR^n$ is principal (as defined in
Section \ref{sec_alg_grp}).
\end{lem}
\begin{proof}
As outlined in Remark \ref{rem_affine_map}, the orbit map
$\theta_p:G\to F_p$, $g\mapsto g.p$, is an affine map in $g$.

If we express
$g=g(t_1,\ldots,t_k)$ in exponential coordinates, then for
fixed $p,q$ in the same $G$-orbit the equation $g(t_1,\ldots,t_k).p=q$
forms an inhomogeneous system of linear equations with unknowns
$t_1,\ldots,t_k$. It has a unique solution because the action
is free. It is well-known from linear algebra that the
solution to such a system can be expressed by expressions
polynomial in the components of $p$ and $q$.
Hence the map $\beta(q,p)=g_{qp}$ with $g_{qp}.p=q$ is a
morphism and the $G$-action is principal.
\end{proof}

\begin{thm}\label{thm_principal_bundle}
Assume that $ L =\mathrm{Z}_{\Iso(\RR^{n}_{s})}(G )$ acts transitively on $\RR^n$.
The orbit space $\RR^n/G $ is isomorphic to $\RR^{n-k}$ as an affine algebraic
variety,
and there exists an algebraic cross section $\sigma:\RR^n/G \to\RR^n$.
In particular, $\RR^n$ is a trivial algebraic principal $G$-bundle
\begin{equation}
G  \to \RR^n \to \RR^{n-k}.
\label{eq_principal_bundle}
\end{equation}
\end{thm}
\begin{proof}
$\RR^n/G $ exists as an affine variety and 
as such is isomorphic to $\RR^{n-k}$
by Proposition \ref{prop_UUp_quotient} below.
We show $\pi:\RR^n\to\RR^n/G \cong\RR^{n-k}$ is a
principal bundle\index{principal bundle} for $G $:
$\RR^n$ and $\RR^{n-k}$ are smooth (hence normal) varieties.
As a consequence of homogeneity and Rosenlicht \cite[Theorem 10]{rosenlicht1},
$\RR^{n-k}$ can be covered by
Zariski-open sets $W$ such that on each $W$ there exists a local
algebraic cross section\index{cross section}
$\sigma_W:W\to\RR^n$, and $\pi$ is a locally trivial fibration.

The $G $-action is principal,
so for any $p\in\RR^n$ and $g\in G $, the
map $\beta(g.p,p)=g$ defined on the graph of the action is a morphism.
Thus the bundle's coordinate changes by $G$
are morphisms and the bundle is algebraic.
The existence of an algebraic cross section $\sigma:\RR^n/G\to\RR^n$ now follows 
from Theorem \ref{thm_trivial_bundle}.
\end{proof}

If $G$ is the real Zariski closure of the fundamental group
$\Gamma$ of a flat pseudo-Riemannian homogeneous space $M$,
then $G$ is an algebraic Wolf group and we can apply
Theorem \ref{thm_principal_bundle} to its action on $\RR^n$.
We can then take the quotient for the action of $\Gamma$
and obtain the following:

\begin{thm}\label{thm_pb2}
Let $M=\RR^{n}_s/\Gamma$ be a complete flat pseudo-Riemanninan homo\-geneous
manifold and $k=\rk\Gamma$.
Let $G$ denote the real Zariski closure of $\Gamma$.
Then $M$ is a trivial fiber bundle
\begin{equation}
G /\Gamma \to M \to \RR^{n-k}
\end{equation}
with structure group $G/\Gamma_{\mathrm{n}}$, where $\Gamma_{\mathrm{n}}$
is the largest subgroup of $\Gamma$ which is a normal subgroup
of $G$.
\end{thm}

For the assertion on the structure group see
Steenrod \cite[Theorem 7.4]{steenrod}.

In order to complete the proof of Theorem
\ref{thm_principal_bundle},
it remains to show that $\RR^n/G$ is in fact isomorphic to
$\RR^{n-k}$ as an affine variety
(Proposition \ref{prop_UUp_quotient}).
To this end, we use a
result due to Rosenlicht \cite[Theorem 5]{rosenlicht4} which 
states that any algebraic homogeneous space for $\RR$-defined
unipotent groups
is isomorphic to some $\RR^m$ as an affine variety.
So we need to identify $\RR^n/G$ with a homogeneous space
$U/U'$ for some unipotent algebraic group $U$ with algebraic
subgroup $U'$:
The centralizer $L$ of $G$ is an algebraic subgroup of
$\Iso(\RR^{n}_{s})$. Choose $U$ to be the unipotent radical
of $L$. The subgroup $U'$ we will be the group
defined in (\ref{eq_UFp}) below.

Because the reductive part of $L$ always has a fixed point
(Baues \cite[Lemma 2.2]{baues}), transitivity of $L$
depends on $U$:

\begin{prop} \label{prop_unirad_transitive}
The centralizer $ L $ acts transitively on $\RR^{n}$ if and only if
its uni\-potent radical $ U \subset  L $ acts transitively.
\end{prop} 

In the following, we assume that $L$ (hence $U$) acts transitively.
For $p\in\RR^n$, its stabilizer $U_p$ is an algebraic subgroup of $U$.

\begin{prop} \label{prop_Rn_quotient}
The quotient $ U / U _p$ is isomorphic to $\RR^n$ as an affine
variety.
\end{prop} 
\begin{proof} 
The quotient $U/U_p$ is an affine variety because $U$ is unipotent (see Borel \cite[Corollary 6.9]{borel}).
As $\RR^{n}$ is a geometric quotient for the action of
$U _p$ on $U$, it is isomorphic to $U/U _p$ as an
affine variety.
\end{proof}

Let $F_p$ denote the $G$-orbit through $p$.
It is helpful to relate the action of $G$ on $\RR^n$ to the
right-action on $U/U_p$ by the group
\begin{equation}
 U _{F_p} = \{ u\in  U  \mid u.F_p\subseteq F_p\}.
\label{eq_UFp}
\end{equation}

\begin{prop} \label{prop_RFp}
$ U _{F_p}$ is an algebraic subgroup of $ U $ acting transitively on
$F_p$, and $U_p$ is a normal subgroup of $U_{F_p}$.
\end{prop} 
\begin{proof} 
$U_{F_p}$ acts transitively on $F_p$ because $U$ acts transitively.
The orbit $F_p$ is Zariski-closed because $G$ is unipotent by Lemma \ref{lem_wolf1}.
So $U_{F_p}$ is an algebraic group as it is the preimage of $F_p$ under
the orbit map $U\to\RR^n$, $u\mapsto u.p$.

As $U_p$ commutes with $G$, it fixes every point of $F_p$.
Hence it is invariant under conjugation with $U_{F_p}$.
\end{proof}

Fix an element $p\in\RR^n$ and let $\tilde{U}=U_{F_p}/U_p$.
The $G$-action is free, so for $u U_p\in\tilde{U}$ there exists a unique element $g_u\in G$
satisfying $u.p=g_u.p$. Then the map
\begin{equation}
\Phi: \tilde{U} \to  G ,\quad u U _p\mapsto g_u^{-1}
\label{eq_isomorphism}
\end{equation}
is an isomorphism of algebraic groups. 

By Proposition \ref{prop_Rn_quotient} there exists an
isomorphism of affine varie\-ties
\[
\Psi: U / U _p\to\RR^n,\quad uU_p\mapsto u.p.
\]

\begin{lem}\label{lem_orbits_equivalent}
$\Psi$ induces a bijection from the orbits of the right-action of $\tilde{ U }$
on $ U / U _p$ to the orbits of $G $ on $\RR^n$.
\end{lem}
\begin{proof}
Let $\tilde{u}\in\tilde{ U }$ and $g=\Phi(\tilde{u})$.
For any $u\in U $,
$\Psi(u U _p)=u.p\in\RR^n$, and
$u U _p.\tilde{u}=u\tilde{u} U _p$ maps to
$\Psi(u\tilde{u} U _p)=u\tilde{u}.p$.
By definition of $\Phi$ and because $U$ centralizes $G$,
\[
u\tilde{u}.p = ug^{-1}.p = g^{-1} u.p.
\]
So $\Psi$ maps the orbit $u U _p.\tilde{ U }$ to the orbit $G .(u.p)$.
\end{proof}

\begin{remark}
Note that $\Psi$ is not equivariant with respect to the action
of $\tilde{U}=G$, but rather \emph{anti-equivariant}.
By this we mean that
$\Psi(\overline{u}.\tilde{u})=\Phi(\tilde{u})^{-1}.\Psi(\overline{u})$ holds
for all $\tilde{u}\in\tilde{U}$ and $\overline{u}\in U/U_p$.
\end{remark}

\begin{prop}\label{prop_UUp_quotient}
$U/U_{F_p}$ is a geometric quotient for the action of $G $ on
$\RR^n$
and isomorphic to $\RR^{n-k}$ as an affine algebraic
variety.
\end{prop}
We will write $\RR^n/G $ for the quotient $ U / U _{F_p}$.
\begin{proof}
$ U / U _{F_p}$ is an algebraic homogeneous space
for a unipotent group. By a theorem of Rosenlicht \cite[Theorem 5]{rosenlicht4},
the quotient $ U / U _{F_p}$ of $\RR$-defined unipotent groups is algebraically isomorphic to an affine space $\RR^m$.
Moreover, $ U / U _{F_p}=( U / U _p)/( U _{F_p}/ U _p)=( U / U _p)/\tilde{ U }$,
so
\[
\dim U / U _{F_p}=\dim U / U_p-\dim\tilde{ U }=\dim\RR^n-\dim G =n-k.
\]
Let $\pi_0: U / U _p\to U / U _{F_p}$ denote the quotient map.
So we have morphisms
\[
\xymatrixcolsep{4pc}\xymatrix{
 U / U _p \ar[d]_{\pi_0} & \RR^n \ar[l]_{\Psi^{-1}}\ar@{-->}[ld]^{\pi} \\
 U / U _{F_p} \ar@{=}[d] \\
\RR^{n-k}
}
\]
where we define $\pi=\pi_0\circ\Psi^{-1}$.
Since $\Psi$ is an isomorphism and $\pi_0$ a quotient map, the map
$\pi$ is a surjective open morphism.
So $\pi$ is a quotient map by \cite[Lemma 6.2]{borel}.
Let $\overline{q}=u U _{F_p}\in U / U _{F_p}$, where $q=u.p$.
Then the fiber of $\pi$ over $\overline{q}$ is
\[
\pi^{-1}(\overline{q})=\Psi(\pi_0^{-1}(\overline{q}))=\Psi(u U _p. U _{F_p})
=G.(u.p)=F_q,
\]
the orbit of $G $ through $q$ (use Lemma \ref{lem_orbits_equivalent} for the
third equality).
Hence $ U / U _{F_p}$ is a geometric quotient for the $G $-action.
\end{proof}

\appendix
\section{Cross Sections for Unipotent Group Actions}\label{sec_algebraic}

In this appendix we show that an affine algebraic principal bundle for
a uni\-potent algebraic group is a trivial bundle.
This result is known, see for example Kraft and Schwarz \cite[Proposition IV.3.4]{KS},
but we give a proof for the reader's convenience and to ensure we
can take all cross sections to be defined over the real numbers.

In the following, an \emph{algebraic principal bundle} will mean a
$G$-principal bundle $\pi:V\to W$ where $V$ and $W$ are smooth affine
varieties, $G$ is an algebraic group acting principally on $V$
(as defined in Section \ref{sec_alg_grp})
such that the bundle's coordinate changes are algebraic maps.
For the applications in this article, we assume all varieties and 
morphisms to be defined over $\RR$ (see also Remark \ref{rem_R_def}).

\begin{lem}\label{lem_trivial_bundle0}
Let $V,W$ be smooth affine varieties and $\pi:V\to W$ an algebraic
principal bundle for a uni\-potent action of the additive group
$\mathrm{G}_{\mathrm{a}}$.
Then there exists an algebraic cross section $\sigma:W\to V$.
\end{lem}
\begin{proof} 
Because $W$ is affine,
it can be covered by a finite system $(U_i)_{i=1}^m$ of dense
open subsets admitting local cross sections $\sigma_i:U_i\to V$.

The action of $\mathrm{G}_{\mathrm{a}}$ is principal, which means the map
$\beta(g.p,p)=g$ defined on the graph of the action is a morphism
(Borel \cite[1.8]{borel}).
But $\mathrm{G}_{\mathrm{a}}=(\RR,+)$, so $\beta$ is in fact a regular real function
on its domain of definition. Hence we can define regular real functions
$\beta_{ij}$ on each $U_{ij}=U_i\cap U_j$ by
\[
\beta_{ij}:U_i\cap U_j\to \RR,\quad
p \mapsto \beta(\sigma_i(p),\sigma_j(p))
\]
satisfying
\[
\sigma_i|_{U_{ij}}(p)=\beta_{ij}(p).\sigma_j|_{U_{ij}}(p).
\]
These $\beta_{ij}$ form a 1-cocycle in the \v{C}ech cohomology
of the sheaf $\mathcal{O}_W$  of $\CC$-valued
regular functions on $W$.
As $W$ is affine, its first \v{C}ech cohomology group for 
$\mathcal{O}_W$ vanishes
(see for example Perrin \cite[Chapter VII, Theorem 2.5]{perrin}).
So there exist $\CC$-valued regular functions $\alpha_i$ defined on $U_i$ such that
\[
\beta_{ij} = \alpha_i|_{U_{ij}} - \alpha_j|_{U_{ij}}.
\]
The maps
\[
p\mapsto-\alpha_i(p).\sigma_i(p),\quad p\mapsto-\alpha_j(p).\sigma_j(p)
\]
are local cross sections defined on $U_i$, $U_j$, respectively,
which coincide on the open set $U_{ij}$. By continuity, they define a
morphism
\[
\sigma^{ij}:U_i\cup U_j\to V
\]
such that $\pi\circ\sigma^{ij}$ coincides with the identity on the open
subset $U_i\subset U_i\cup U_j$.
Hence $\pi\circ\sigma^{ij}=\id_{U_i\cup U_j}$ and $\sigma^{ij}$
is a local cross section defined on $U_i\cup U_j$.

Finitely many repetitions of this yield a global cross section
$\sigma:W\to V$.
\end{proof}

\begin{lem}\label{lem_trivial_bundle1}
Let $V,W$ be smooth affine varieties and $\pi:V\to W$ an algbraic
principal $G$-bundle for a unipotent algebraic group $G$.
Let $ H , A \subset G $ be as in (\ref{eq_malcev_opg}).
Then $V$ is also an algebraic
principal $ H $-bundle with base $W\times A $.
\end{lem}
\begin{proof} 
Recall that $ A \cong\mathrm{G}_{\mathrm{a}}$ and $ A \cong G / H $ as algebraic groups.
Thus there is a cross section $ G / H \to A \hookrightarrow G $.
It then follows that the quotient $V/ H $ exists as an affine variety,
$V\to V/ H $ is an algebraic principal $ H $-bundle, and that $V/ H \to W$ is a
bundle with structure group $ A $. By Lemma \ref{lem_trivial_bundle0},
$V/ H \cong W\times A $ as an algebraic principal $ A $-bundle.
\end{proof}

\begin{thm}\label{thm_trivial_bundle}
Let $V,W$ be smooth affine varieties and $\pi:V\to W$ an algebraic 
principal bundle for a unipotent algebraic group $G$.
Then $W=V/ G $ and there exists an algebraic cross section
$\sigma:V/ G \to V$.
\end{thm}
\begin{proof} 
Let $k=\dim G $. The case $k=1$ is Lemma \ref{lem_trivial_bundle0}.
The theorem follows by induction on $k$:
Let $ H $, $ A $ denote the subgroups from (\ref{eq_malcev_opg}).

By Lemma \ref{lem_trivial_bundle1} we may apply the induction hypothesis
to $ H $, and together with Lemma \ref{lem_trivial_bundle0}, we have
global cross sections:
\begin{align*}
\sigma_{ H }&:V/ H  \to V,\\
\sigma_{ A }&:(V/ H )/ A  \to V/ H .
\end{align*}
$(V/ H )/ A =(V/ H )/( G / H )=V/G$ as affine varieties
(Borel \cite[Corollary 6.10]{borel}).
Define a morphism
$\sigma=\sigma_{ H }\circ\sigma_{ A }:V/ G \to V$.
For any orbit $G.p$ we have
\[
(\pi\circ\sigma)(G.p) = (\pi\circ\sigma_{ H }\circ\sigma_{ A })( G .p)
= (\pi\circ\sigma_{ H }\circ\sigma_{ A })( A .( H .p))
= G .p.
\]
So $\pi\circ\sigma=\id_{V/ G }$, that is $\sigma$ is a global cross section
for the action of $ G $.
\end{proof}

\begin{remark}\label{rem_R_def}
As we assumed all varieties and all morphisms to be defined over $\RR$,
the cross section $\sigma$ may be taken to be defined over $\RR$:
In the proof of Theorem
\ref{thm_trivial_bundle}, we may assume $\sigma_{ H }$ to be $\RR$-defined
by the induction hypothesis. Further, $\sigma_{ A }$ may be assumed to be
$\RR$-defined, because in the proof of Lemma \ref{lem_trivial_bundle0},
the local cross sections $\sigma_i$ can be assumed to be
$\RR$-defined by Rosenlicht's results on solvable algebraic 
groups
\cite[Theorem 10]{rosenlicht1}, and the 1-cocycles
$\alpha_i$ may be replaced by their real parts and still yield
$\alpha_i-\alpha_j=\beta_{ij}$, because the latter is an $\RR$-valued regular
function which is defined over $\RR$ if the action of $ A $ is.
\end{remark}


\begin{thebibliography}{99}

\bibitem{baues} O. Baues, {\em Flat pseudo-Riemannian manifolds and prehomogeneous affine representations}, 
in 'Handbook of Pseudo-Riemannian Geometry and Supersymmetry', EMS, IRMA Lect. Math. Theor. Phys. {16}, 2010, pp. 731-817
(also arXiv:0809.0824v1)

\bibitem{BG} O. Baues, W. Globke, {\em Flat Pseudo-Riemannian Homogeneous Spaces With Non-Abelian Holonomy Group},
Proc. Amer. Math. Soc. 140, 2012, pp. 2479-2488
(also arXiv:1009.3383)

\bibitem{borel} A. Borel,
{\em Linear Algebraic Groups}, 2nd edition,
Springer, 1991

\bibitem{corwin} L. Corwin, F.P. Greenleaf,
{\em Representations of nilpotent Lie groups and their applications},
Cambridge University Press, 1990

\bibitem{globke} W. Globke, {\em Holonomy Groups of Flat Pseudo-Riemannian Homogeneous Manifolds}, Dissertation, Karlsruhe Institute of Technology, 2011

\bibitem{globke2} W. Globke, {\em Holonomy Groups of Complete Flat Pseudo-Riemannian Homogeneous Spaces}, Adv. in Math. 240, 2013, pp. 88-105 (also arXiv:1205.3285)

\bibitem{KS} H. Kraft, G.W. Schwarz,
{\em Reductive Group Actions with One-Dimensional Quotient},
Publ. Math. Inst. Hautes \'Etudes Sci. 76, 1992, pp. 1-97

\bibitem{oneill} B. O'Neill,
{\em Semi-Riemannian Geometry},
Academic Press, 1983

\bibitem{OV} A. Onishchik, E. Vinberg,
{\em Lie Groups and Lie Algebras III},
Springer, 1994

\bibitem{perrin} D. Perrin
{\em Algebraic Geometry},
Springer, 2008

\bibitem{ragh} M.S. Raghunathan,
{\em Discrete Subgroups of Lie Groups},
Springer, 1972

\bibitem{rosenlicht1} M. Rosenlicht,
{\em Some Basic Theorems on Algebraic Groups},
Amer. J. Math. 78, no. 2, 1956, pp. 401-443



\bibitem{rosenlicht4} M. Rosenlicht,
{\em Questions of Rationality for Solvable Algebraic Groups over Non\-perfect Fields},
Ann. Mat. Pura Appl. 62, no. 1, 1963, pp. 97-120

\bibitem{steenrod} N. Steenrod,
{\em The Topology of Fibre Bundles},
Princeton Univ. Press, 1951

\bibitem{Wolf_1} J.A. Wolf, {\em Homogeneous manifolds of zero curvature},
Trans. Amer. Math. Soc. 104, 1962, pp. 462-469

\bibitem{Wolf_buch} J.A. Wolf,
{\em Spaces of Constant Curvature}, 6th edition,
Amer. Math. Soc., 2011


\end{thebibliography}
\bibliographystyle{ijmart}

\end{document}